\documentclass{amsart}
\usepackage{xspace}
\usepackage{graphicx} % ei kasuta
\newtheorem{thm}{Theorem}[section]

\newtheorem{lemma}[thm]{Lemma}
\newtheorem{proposition}[thm]{Proposition}
\theoremstyle{definition}

\newtheorem*{definition*}{Definition}
\theoremstyle{remark}

\numberwithin{equation}{section}
\newcommand{\eps}{\varepsilon}
\usepackage[dvipsnames,x11names,svgnames]{xcolor} %UUS!
\definecolor{celadon}{rgb}{0.67, 0.88, 0.69}
\colorlet{kollane}{LightGoldenrod1}
\colorlet{roheline}{celadon}
\colorlet{sinine}{LightSkyBlue2}
\colorlet{lilla}{violet}
\colorlet{hall}{lightgray}
\usepackage[framemethod=tikz]{mdframed}
\newcommand{\N}{\mathbb{N}}

\newcommand{\ws}{w^\ast}
\newcommand{\Xs}{X^\ast}
\newcommand{\xs}{x^\ast}
\newcommand{\zs}{z^\ast}
\newcommand{\Zs}{Z^\ast}
\newcommand{\CWO}{\textit{CWO}\xspace}
\newcommand{\CWOS}{\mbox{\textit{CWO}-$S$\xspace}}
\newcommand{\CWOB}{\mbox{\textit{CWO}-$B$\xspace}}
\DeclareMathOperator{\re}{Re}
\DeclareMathOperator{\diam}{diam}
\begin{document}

\title[Convex combinations of relatively weakly open subsets]%
{Norm-one points in convex combinations of relatively weakly open subsets of the unit ball in the spaces $L_1(\mu,X)$}%
%Norm-One Points in Convex Combinations of Relatively Weakly Open Subsets of the Unit Ball in $L_1(\mu, X)$ %vector-valued $L_1$-spaces
\author{Rainis Haller}%
\author{Paavo Kuuseok}%
\author{M\"{a}rt P\~{o}ldvere}%
\address[R.~Haller and M.~P\~{o}ldvere]%
{Institute of Mathematics and Statistics, University of Tartu, Narva mnt 18, 51009, Tartu, Estonia}%%
\email{rainis.haller@ut.ee, mart.poldvere@ut.ee}%

\address[P. Kuuseok]%
{Kuressaare College, School of Engineering, Tallinn University of Technology,
Tallinna 19, 93811, Kuressaare, Estonia}%%
\email{paavo.kuuseok@taltech.ee}%

\thanks{Research supported by the Estonian Research Council grant PRG1901}%

%\date{}%
%\dedicatory{}%
%\commby{}%
% ----------------------------------------------------------------
\begin{abstract}
In a paper published in 2020 in \textsl{Studia Mathematica,}
Abrahamsen et al. proved that in the real space $L_1(\mu)$,
where $\mu$ is a non-zero $\sigma$-finite (countably additive non-negative) measure,
norm-one elements in finite convex combinations of relatively weakly open subsets of the unit ball
are interior points of these convex combinations in the relative weak topology.
In this paper that result is generalised by proving that
the same is true in the (real or complex) Lebesgue--Bochner spaces $L_1(\mu,X)$
where $X$ is a weakly uniformly rotund Banach space.
\end{abstract}

% \begin{abstract}
% In a 2020 paper in \textit{Studia Mathematica}, Abrahamsen et al.\ proved that in the real space $L_1(\mu)$, where $\mu$ is a nonzero $\sigma$-finite measure, every norm-one element in a convex combination of relatively weakly open subsets of the unit ball is an interior point of that combination. We generalise this result by proving that the same holds in the (real or complex) Lebesgue–-Bochner spaces $L_1(\mu, X)$ for every weakly uniformly rotund Banach space $X$.
% \end{abstract}

\maketitle

% ----------------------------------------------------------------

%%%%%%%%%%%%%%%%%%%%%%%%%%%%%%%%%%%%%%%%%%%%%%%%%%%%%%%%%%%%%%%%%%%%%%
\section{Introduction}
%%%%%%%%%%%%%%%%%%%%%%%%%%%%%%%%%%%%%%%%%%%%%%%%%%%%%%%%%%%%%%%%%%%%%%

For a Banach space $X$, we denote by $S_X$, $B_X$, and $B_X^\circ$, respectively,
its unit sphere, closed unit ball, and open unit ball.
The topological dual of $X$ is denoted by $X^\ast$.
%Throughout the paper,
Whenever dealing with the space $L_1(\mu,X)$ where $\mu$ is a non-negative measure,
we assume the measure $\mu$ to be complete.

\medskip
In the memoir \cite{GGMS} by Ghoussoub, Godefroy, Maurey, and Schachermayer, it was proven that
\emph{in the positive face $\mathcal{F}:=\{f\in L_1[0,1]\colon f\geq0,\,\|f\|=1\}$
of the closed unit ball of (the real space) $L_1[0,1]$,
every (finite) convex combination of relatively weakly open subsets (in particular, slices)
is still weakly open} \cite[page 48, Remark IV.5]{GGMS}.
This property of $\mathcal{F}$ was called ``remarkable'' in \cite{GGMS}, as, in general, a convex combination of relatively weakly open subsets need not be relatively weakly open.
(For instance, in strictly convex Banach spaces,
no non-trivial convex combination of two disjoint subsets of the closed unit ball intersects the unit sphere;
therefore such combinations cannot be relatively weakly open in the closed unit ball.)
Another example of such a phenomenon was discovered by Abrahamsen and Lima
in \cite[Theorems 2.3 and 2.4]{AL}
where it was proven that \emph{in the (real or complex) spaces C(K), where $K$ is a scattered compact Hausdorff space,
and $c_0$, convex combinations of slices of the unit ball are relatively weakly open.}
On the other hand, in \cite[Theorem 3.1 and Remark 3.1]{HKP} it was proven that
whenever the space $C_0(L)$, where $L$ is a locally compact Hausdorff space,
has this property, then the space $L$ is scattered. 
Thus, in fact, the space $C_0(L)$ has this property if and only if the space~$L$ is scattered
(see \cite[Theorem~3.1]{ABHLP}).
An extension of this result for the spaces $C_0(L,\mathcal A)$, where $\mathcal A$ is a compact $C^\ast$-algebra, was obtained by Becerra Guerrero and Fernández-Polo in \cite[Theorem~4.3]{GP}.
An ultimate version of this result in the context of $L_1$-predual spaces
was proven by López-Pérez and Medina in \cite[Theorem~2.3 and Corollary 2.4]{L-PM}.
The results in \cite{L-PM} involve the notion of \emph{stability.}
A convex subset $C$ of a topological vector space is said to be \emph{stable}
%if the mapping $C\times C\ni(x,y)\mapsto \frac{x+y}{2}\in C$ is open 
%with respect to the respective relative topologies.
%Equivalently, the set $C$ is stable if and only 
if every convex combination of relatively open subsets in $C$
is open in the relative topology of $C$ (see, e.g., \cite[Proposition 1.1]{CP}, 
for equivalent reformulations for the notion of stability).
So, the above-mentioned result in \cite{GGMS} says that the positive face $\mathcal{F}$ is stable
with respect to its relative weak topology.
We refer to \cite[Introduction]{L-PM} for an overview of various results in Banach spaces that involve the notion of stability.

In \cite{ABHLP}, the following notions were defined and a thorough study of these notions was taken on.
\begin{definition*}[see {\cite[Definition 1.1]{ABHLP}}]%\label{def: CWO-s}
A Banach space $X$ is said to have
\begin{itemize}
\item%[\textup{(a)}]
\emph{property \CWO} if, for every convex combination $C$ of relatively weakly open subsets of $B_X$,
the set $C$ is open in the relative weak topology of $B_X$;
\item%[\textup{(b)}]
\emph{property \CWOS} if, for every convex combination $C$ of relatively weakly open subsets of $B_X$,
every $x \in C\cap S_X$ is an interior point of~$C$ in the relative weak topology of $B_X$;
\item%[\textup{(c)}]
\emph{property \CWOB} if, for every convex combination $C$ of relatively weakly open subsets of $B_X$,
every $x \in C\cap B_X^\circ$ is an interior point of~$C$ in the relative weak topology of $B_X$.
\end{itemize}
\end{definition*}
\noindent%
Thus, \emph{a Banach space has property \CWO{} if and only if it has both properties \CWOS{} and \CWOB{}}.
Note that property \CWO{} for $X$ means precisely that $B_X$ is stable with respect to its relative weak topology.
In \cite{ABHLP}, a geometric property for Banach spaces $X$---called property $(co)$ in that paper---guaranteeing that,
for finite dimensional $X$ and scattered $L$, the space space $C_0(L,X)$ has property \CWO{}, was singled out
(see \cite[Definition 2.1 and Theorem 2.5]{ABHLP}).
For finite-dimensional $X$, property $(co)$ is formally stronger than property \CWO{} \cite[Proposition 2.2]{ABHLP}.
In \cite[Proposition 3.7]{L-PM}, it was proven that $X$ has property $(co)$ if and only if its unit ball $B_X$ is stable with respect to its relative norm topology.
It follows that, for finite-dimensional $X$, properties $(co)$ and \CWO{} are equivalent
(this was proven independently by Kadets in \cite{Kadets_2019}),
and thus, for scattered $L$, the space $C_0(L,X)$ has property \CWO{} 
if and only if $X$ has property \CWO{}.

In \cite[Theorem~5.5]{ABHLP}, it was proven that 
\emph{the real space $L_1(\mu)$,
where $\mu$ is a non-zero $\sigma$-finite (countably additive non-negative) measure,
has property \CWOS}.
(Note that no $1$-sum of two infinite-dimensional Banach spaces can have property \CWO{}  \cite[Proposition 2.1]{HKP}.)
The objective of the present paper is to generalise this result by proving the following theorem.
Recall that a Banach space $X$ is said to be \emph{weakly uniformly rotund} (in brief, \emph{$w$UR})
if, for every $\xs\in\Xs$ and every $\eps>0$, there is a $\delta>0$ such that, whenever $x,y\in S_X$,
\[
\bigl\|\tfrac{1}{2}(x+y)\bigr\|>1-\delta
\quad\Longrightarrow\quad
|\xs(x-y)|<\eps.
\]

\begin{thm}\label{thm: L_1(mu,X) has property CWO-S if X is wUR}
Let $\mu$ be a non-zero $\sigma$-finite (countably additive non-negative) measure
on a $\sigma$-algebra $\Sigma$ of subsets of a set $\Omega$,
and let $X$ be a $w$UR Banach space.
Then the space $L_1(\mu,X)$ has property \CWOS.
\end{thm}

The paper is organised as follows.

Section 2 contains auxiliary results
for the proof of Theorem \ref{thm: L_1(mu,X) has property CWO-S if X is wUR}
(Lemmata \ref{lemma: one may assume that the sets I_A and I_B are empty}--\ref{lemma: wUR}).
In particular, Lemma \ref{lemma: one may assume that the sets I_A and I_B are empty} 
fills a minor gap in the proof of \cite[Theorem~5.5]{ABHLP}.

In Section 3, we prove the following proposition
that reduces the proof of Theorem~\ref{thm: L_1(mu,X) has property CWO-S if X is wUR}
to the case when the measure $\mu$ is finite.

\begin{proposition}\label{prop: 1-sum of Banach spaces with property CWO-S has property CWO-S}
For every $i\in\N$, let $X_i$ be a Banach space with property \CWOS.
Then also the $1$-sum $\mathop{\bigoplus_1}\limits_{i=1}^\infty X_i$ has property \CWOS.
\end{proposition}

Finally, in Section 4, we prove Theorem~\ref{thm: L_1(mu,X) has property CWO-S if X is wUR}.

\medskip
\textbf{A natural question} arises in connection with Theorem \ref{thm: L_1(mu,X) has property CWO-S if X is wUR}:
does this theorem remain true when $X$ is assumed to be any Banach space with property \CWOS{}
(note that strictly convex spaces---and thus, as a special case, wUR spaces---do have property \CWOS)
or, moreover, does the injective tensor product of two Banach spaces with property \CWOS{} have property \CWOS?

%\subsection*{\indent A natural question\nopunct} arises in

%%%%%%%%%%%%%%%%%%%%%%%%%%%%%%%%%%%%%%%%%%%%%%%%%%%%%%%%%%%%%%%%%%%%%%
\section{Auxiliary results for the proof of Theorem \ref{thm: L_1(mu,X) has property CWO-S if X is wUR}}
%%%%%%%%%%%%%%%%%%%%%%%%%%%%%%%%%%%%%%%%%%%%%%%%%%%%%%%%%%%%%%%%%%%%%%

In \cite[Theorem~5.5]{ABHLP} it was proven that,
for any non-zero $\sigma$-finite (countably additive non-negative) measure $\mu$,
the real space $L_1(\mu)$ has property $\CWOS$.
However, the proof in \cite{ABHLP} contains a gap: it is valid only for the case when,
using the notation of that proof, the set $J_1$ is non-empty.
Nevertheless, the following lemma
(that will also be used in the proof of Theorem \ref{thm: L_1(mu,X) has property CWO-S if X is wUR})
shows that, in the proof of \cite[Theorem~5.5]{ABHLP},
it actually suffices to consider only this very case,
and thus the result holds.

\begin{lemma}\label{lemma: one may assume that the sets I_A and I_B are empty}
Let $\mu$ be a non-zero (countably additive non-negative) measure
on a $\sigma$-algebra $\Sigma$ of subsets of a set $\Omega$,
and let $X$ be a Banach space.
If
\begin{itemize}
\item[$(\sharp)$]
whenever $x,y,z\in S_{L_1(\mu,X)}$ with
\[
\bigl\{\omega\in\Omega\colon\;
\text{$x(\omega)\not=0$ and $y(\omega)=0$, or $x(\omega)=0$ and $y(\omega)\not=0$}\bigr\}
=\emptyset
\]
and real numbers $\lambda,\nu\in(0,1)$ with $\lambda+\nu=1$
are such that $\lambda x+\nu y=z$ (in the space $L_1(\mu, X)$),
and $U$ and $V$ are neighbourhoods of, respectively, $x$~and~$y$
in the relative weak topology of $B_{L_1(\mu,X)}$,
there is a neighbourhood~$W$ of $z$
in the relative weak topology of $B_{L_1(\mu,X)}$ such that $W\subset \lambda U+\nu V$,
\end{itemize}
then the space $L_1(\mu, X)$ has property \CWOS.
\end{lemma}
\begin{proof}
Assume that $(\sharp)$ holds,
and let relatively weakly open subsets $U$ and $V$ of $B_{L_1(\mu,X)}$,
real numbers $\lambda,\nu\in(0,1)$ with $\lambda+\nu=1$, and $z\in S_{L_1(\mu,X)}$
be such that $z\in\lambda U+\nu V$.
By \cite[Lemma~4.1, (b)]{ABHLP},
it suffices to find a neighbourhood $W$ of $z$ in the relative weak topology of $B_{L_1(\mu,X)}$
such that $W\subset\lambda U+\nu V$.

Let $a\in U$ and $b\in V$ be such that $z=\lambda a+\nu b$.
Since the sets $U$ and $V$ are relatively norm open in $B_{L_1(\mu, X)}$,
there is a real number $\delta$ with $0<\delta<1$ such that
\[
B(a,2\delta)\cap B_{L_1(\mu, X)}\subset U
\quad\text{and}\quad
B(b,2\delta)\cap B_{L_1(\mu, X)}\subset V.
\]
Defining
\[
x:=(1-\nu\delta)a+\nu\delta b=a+\nu\delta(b-a)
\]
and
\[
y:=\lambda\delta a+(1-\lambda\delta)b=b-\lambda\delta(b-a),
\]
one has
$x\in S_{L_1(\mu,X)}\cap B(a,2\delta)\subset U$ and $y\in S_{L_1(\mu,X)}\cap B(b,2\delta)\subset V$
with $\lambda x+\nu y=z$ in the space $L_1(\mu,X)$.
Indeed, in order to see that $x,y\in  S_{L_1(\mu,X)}$,
observe that, since $\|\lambda a+\nu b\|=\|z\|=\|\lambda a\|+\|\nu b\|$,
one has $\|\alpha a+\beta b\|=\alpha\|a\|+\beta\|b\|$ whenever $\alpha,\beta\geq0$.
By a similar argument, the set
\[
H:=\{\omega\in\Omega\colon \text{$x(\omega)\not=0$ and $y(\omega)=0$, or $x(\omega)=0$ and $y(\omega)\not=0$}\}
\]
is of $\mu$-measure zero.
(Here the key observation is that,
since $\|\lambda a(\omega)+\nu b(\omega)\|=\|\lambda a(\omega)\|+\|\nu b(\omega)\|$ for a.e. $\omega\in\Omega$,
one has $\|\alpha a(\omega)+\beta b(\omega)\|=\alpha\|a(\omega)\|+\beta\|b(\omega)\|$ whenever $\alpha,\beta\geq0$
for a.e. $\omega\in\Omega$.)
Redefining the values of $a$, $b$, $x$, $y$, and~$z$ on~$H$ to become $0$,
the redefined $a$, $b$, $x$, $y$, and~$z$ are the same functions in $L_1(\mu, X)$ as the original ones.
The existence of the desired~$W$ now follows from $(\sharp)$.
\end{proof}

The following two lemmata may be known; however, we do not know of any references for them.

\begin{lemma}\label{lemma: L_1(mu,X)}
Let $\mu$ be a non-zero finite (countably additive non-negative) measure
on a $\sigma$-algebra $\Sigma$ of subsets of a set $\Omega$, and let $X$ be a Banach space.
Let $z\in S_{L_1(\mu,X)}$, let $E\in\Sigma$, and let $0<\eps<1$.
\begin{itemize}
\item[\textup{(a)}]
If $z(\omega)\not=0$ for every $\omega\in E$,
then there are $n\in\N$ and pairwise disjoint measurable subsets $E_0,E_1,\dotsc,E_n$ of $E$
with $\bigcup_{i=0}^n E_i=E$ such that
\begin{itemize}
\item[\textup{(1)}]
$\mu(E_0)<\eps$;
\item[\textup{(2)}]
$\bigl\|\int_{E_i}z\,d\mu\bigr\|>(1-\eps)\int_{E_i}\|z\|\,d\mu$
for every $i\in\{1,\dotsc,n\}$ with $\mu(E_i)>0$;
\item[\textup{(3)}]
for every $i\in\{1,\dotsc,n\}$ with $\mu(E_i)>0$,
there is a neighbourhood~$W_i$ of $z$ in the relative weak topology of $B_{L_1(\mu,X)}$ such that
\begin{equation}\label{eq: int_(E_i) ||w||dmu > (1-eps)int_(E_i) ||z||dmu for every w in W_i}
\biggl\|\int_{E_i} w\,d\mu\biggr\|>(1-\eps)\int_{E_i}\|z\|\,d\mu
\quad\text{for every $w\in W_i$;}
\end{equation}
this neighbourhood can be chosen to be of the form
$W_i:=\bigl\{w\in B_{L_1(\mu,X)}\colon \bigl|\int_{E_i}\ws_i(w-z)\,d\mu\bigr|<\delta_i\bigr\}$
for some $\ws_i\in S_{\Xs}$ and $\delta_i>0$.
\end{itemize}
\item[\textup{(b)}]
There is a neighbourhood~$W$ of $z$ in the relative weak topology of $B_{L_1(\mu,X)}$ such that
\begin{equation}\label{eq: int_E ||z||dmu -eps < int_E ||w||dmu < int_E ||z||dmu +eps for every w in W}
\int_{E}\|z\|\,d\mu-\eps<\int_{E}\|w\|\,d\mu<\int_{E}\|z\|\,d\mu+\eps \quad\text{for every $w\in W$.}
\end{equation}
\end{itemize}
\end{lemma}
\begin{proof}
(a).
Assume that $z(\omega)\not=0$ for every $\omega\in E$.
Choose a real number $\sigma>0$ so that $\mu(E\setminus D)<\frac{\eps}{2}$
where $D:=\{\omega\in\Omega\colon \|z(\omega)\|>\sigma\}$.
Since the function $z$ is essentially separably-valued,
there is a subset $N$ of $D$ such that $\mu(N)=0$ and the set $z(D\setminus N)$ is separable.
It follows that there are pairwise disjoint measurable subsets $E_i$ of $E$, $i=1,2,\dotsc$,
such that $\diam z(E_i)<\frac{\eps\sigma}{2}$ for every $i\in\N$,
and $D\setminus N=\bigcup_{i=1}^\infty E_i$.
Pick an $n\in\N$ so that $\mu(C)<\frac{\eps}{2}$ where $C:=\bigcup_{i=n+1}^\infty E_i$,
and set $E_0:=(E\setminus D)\cup N\cup C$.

Fix an $i\in\{1,\dotsc,n\}$ and suppose that $\mu(E_i)>0$.
Letting $z_i\in z(E_i)$ be arbitrary and picking a $\zs_i\in S_{\Xs}$ so that $\re\zs_i(z_i)=\|z_i\|$,
one has
\begin{equation*}%\label{eq: ||int_(E_i) z dmu|| > (1-eps) int_(E_i) ||z|| dmu}
\begin{aligned}
\biggl\|\int_{E_i} z\,d\mu\biggr\|
&\geq\re\zs_i\biggl(\int_{E_i} z\,d\mu\biggr)
=\int_{E_i}\bigl(\re\zs_i(z_i)+\re\zs_i(z-z_i)\bigr)\,d\mu\\
&\geq\int_{E_i}(\|z_i\|-\|z-z_i\|)\,d\mu
\geq\int_{E_i}(\|z\|-2\|z-z_i\|)\,d\mu\\
&>\int_{E_i}(\|z\|-\eps\sigma)\,d\mu
\geq\int_{E_i}(\|z\|-\eps\|z\|)\,d\mu
=(1-\eps)\int_{E_i}\|z\|\,d\mu.
\end{aligned}
\end{equation*}
Now pick a $\ws_i\in S_{\Xs}$ satisfying $\re\ws_i\bigl(\int_{E_i}z\,d\mu\bigr)=\bigl\|\int_{E_i}z\,d\mu\bigr\|$
and set $W_i:=\bigl\{w\in B_{L_1(\mu,X)}\colon \bigl|\int_{E_i}\ws_i(w-z)\,d\mu\bigr|<\delta_i\bigr\}$
where $\delta_i:=\bigl\|\int_{E_i}z\,d\mu\bigr\|-(1-\eps)\int_{E_i}\|z\|\,d\mu>0$.
For every $w\in W_i$, one has
\begin{align*}
\biggl\|\int_{E_i}w\,d\mu\biggr\|
&\geq\re\ws_i\biggl(\int_{E_i}w\,d\mu\biggr)
=\re\ws_i\biggl(\int_{E_i}z\,d\mu\biggr)+\re\biggl(\int_{E_i}\ws_i(w-z)\,d\mu\biggr)\\
&\geq\biggl\|\int_{E_i}z\,d\mu\biggr\|-\biggl|\int_{E_i}\ws_i(w-z)\,d\mu\biggr|\\
&>\biggl\|\int_{E_i}z\,d\mu\biggr\|-\delta_i
=(1-\eps)\int_{E_i}\|z\|\,d\mu.
\end{align*}

(b).
We first find a neighbourhood $W'$ of $z$ in the relative weak topology of~$B_{L_1(\mu,X)}$
such that the first inequality in \eqref{eq: int_E ||z||dmu -eps < int_E ||w||dmu < int_E ||z||dmu +eps for every w in W}
holds for every $w\in W'$.
If $\int_E \|z\|\,d\mu=0$, then this inequality holds for every $w\in B_{L_1(\mu,X)}$.
Suppose that $\int_E \|z\|\,d\mu>0$.
Define $D:=\bigl\{\omega\in E\colon z(\omega)\not=0\bigr\}$ and let $\delta>0$ be such that
$\int_A\|z\|\,d\mu<\frac{\eps}{2}$ whenever $A\in\Sigma$ satisfies $\mu(A)<\delta$.
By part (a) of the lemma, there are $n\in\N$
and pairwise disjoint measurable subsets $E_0,E_1,\dotsc,E_n$ of~$D$
with $\bigcup_{i=0}^n E_i=D$ such that $\mu(E_0)<\delta$ and, for every $i\in\{1,\dotsc,n\}$ with $\mu(E_i)>0$,
there is a neighbourhood~$W_i$ of $z$ in the relative weak topology of $B_{L_1(\mu,X)}$
satisfying \eqref{eq: int_(E_i) ||w||dmu > (1-eps)int_(E_i) ||z||dmu for every w in W_i}
with $\eps$ replaced by $\frac{\eps}{2}$.
For every $i\in\{1,\dotsc,n\}$ with $\mu(E_i)=0$, set $W_i:=B_{L_1(\mu,X)}$,
and define $W':=\bigcap_{i=1}^n W_i$.
For every $w\in W'$, one has, writing $B:=\bigcup_{i=1}^n E_i$,
\begin{align*}
\int_E \|w\|\,d\mu
&\geq\int_B \|w\|\,d\mu
=\sum_{i=1}^n\int_{E_i}\|w\|\,d\mu
\geq\biggl(1-\frac{\eps}{2}\biggr)\sum_{i=1}^n\int_{E_i}\|z\|\,d\mu\\
&=\biggl(1-\frac{\eps}{2}\biggr)\int_B\|z\|\,d\mu
=\int_B\|z\|\,d\mu-\frac{\eps}{2}\int_B\|z\|\,d\mu\\
&\geq\int_B\|z\|\,d\mu-\frac{\eps}{2}
=\int_E\|z\|\,d\mu-\int_{E_0}\|z\|\,d\mu-\frac{\eps}{2}\\
&>\int_E\|z\|\,d\mu-\frac{\eps}{2}-\frac{\eps}{2}
=\int_E\|z\|\,d\mu-\eps
\end{align*}
(note that $\int_{E_0}\|z\|\,d\mu<\frac{\eps}{2}$
because $\mu(E_0)<\delta$).

By what we just proved, there is a neighbourhood $W''$ of $z$ in the relative weak topology of~$B_{L_1(\mu,X)}$
such that $\int_{\Omega\setminus E} \|w\|\,d\mu>\int_{\Omega\setminus E}\|z\|\,d\mu-\eps$ for every $w\in W''$.
For every $w\in W''$, one has
\begin{align*}
\int_E \|w\|\,d\mu
&\leq1-\int_{\Omega\setminus E} \|w\|\,d\mu
<1-\biggl(\int_{\Omega\setminus E}\|z\|\,d\mu-\eps\biggr)
=\int_E \|z\|\,d\mu+\eps.
\end{align*}
It remains to define $W:=W'\cap W''$.
\end{proof}

\begin{lemma}\label{lemma: wUR}
Let $X$ be a $w$UR Banach space, let $\xs\in\Xs$, and let $0<\xi\leq\frac{1}{2}$ and $\eps>0$.
Then there is a real number $\theta\in(0,1)$ such that,
whenever $a,b\in B_X$ and real numbers $\lambda$ and $\nu$ with $\min\{\lambda,\nu\}\geq\xi$ and $\lambda+\nu=1$
satisfy $\|\lambda a+\nu b\|>\theta$, one has $|\xs(a-b)|<\eps$.
\end{lemma}
\begin{proof}
We may and will assume that $\|\xs\|\leq1$.
Pick a real number $\eps_0>0$ so that $\frac{3\eps_0}{\xi}<\eps$.
Since $X$ is $w$UR, there is a real number $\theta_0$ with $0<\theta_0<1$
such that $|\xs(x-y)|<\eps_0$ whenever $x,y\in S_X$ satisfy $\bigl\|\frac{1}{2}(x+y)\bigr\|>\theta_0$.
Pick a real number $\theta$ so that $\max\{\theta_0,1-\nobreak\eps_0\}<\theta<1$.
Now suppose that $a,b\in B_X$ and real numbers $\lambda$ and $\nu$ with $\min\{\lambda,\nu\}\geq\xi$ and $\lambda+\nu=1$
satisfy $\|\lambda a+\nu b\|>\theta$. Write $c:=\lambda a+\nu b$.
Observe that at least one of the following assertions must hold:
\begin{enumerate}
\item[\textup{(1)}]
$\|(1-t)a+tb\|\geq\|c\|$ whenever $0\leq t<\nu$;
\item[\textup{(2)}]
$\|(1-t)a+tb\|\geq\|c\|$ whenever $\nu<t\leq1$.
\end{enumerate}
By symmetry, one may, without loss of generality, assume that \textup{(1)} holds.
(Note that~\textup{(2)} is equivalent to the condition that $\|(1-s)b+sa\|\geq\|c\|$ whenever $0\leq s<\lambda$.)
Defining $x:=\frac{a}{\|a\|}$ and $y:=\frac{c}{\|c\|}$,
one has  $x,y\in S_X$ with $\bigl\|\frac{1}{2}(x+y)\bigr\|\geq\|c\|>\theta>\theta_0$,
thus $|\xs(x-y)|<\eps_0$ by the choice of $\theta_0$.
It follows that
\begin{align*}
\xi|\xs(a-b)|
&\leq\nu|\xs(a-b)|
=|\xs(a-c)|
\leq\|a-x\|+|\xs(x-y)|+\|y-c\|\\
&=1-\|a\|+|\xs(x-y)|+1-\|c\|
<1-\theta +\eps_0 + 1-\theta
<3\eps_0
\end{align*}
whence $|\xs(a-b)|<\frac{3\eps_0}{\xi}<\eps$, as desired.
\end{proof}

%%%%%%%%%%%%%%%%%%%%%%%%%%%%%%%%%%%%%%%%%%%%%%%%%%%%%%%%%%%%%%%%%%%%%%
\section{Proof of  Proposition~\ref{prop: 1-sum of Banach spaces with property CWO-S has property CWO-S}}
%%%%%%%%%%%%%%%%%%%%%%%%%%%%%%%%%%%%%%%%%%%%%%%%%%%%%%%%%%%%%%%%%%%%%%

\begin{proof}[Proof of  Proposition~\ref{prop: 1-sum of Banach spaces with property CWO-S has property CWO-S}]
We adjust the arguments from the proofs
of \cite[Theorem~2.3]{HKP} and \cite[Theorem 5.5]{ABHLP}.
Set $Z:=\mathop{\bigoplus_1}\limits_{i=1}^\infty X_i$,
and let $x,y,z\in S_Z$ and real numbers $\lambda,\nu\in(0,1)$ with $\lambda+\nu=1$
be such that $z = \lambda x + \nu y$.
Let $F$ be a finite subset of~$S_{\Zs}$ and let $0<\eps<1$. Define
\[
U:=\bigl\{u\in B_Z\colon \text{$|f(u-x)|<6\eps$ for every $f\in F$}\big\}
\]
and
\[
V:=\bigl\{v\in B_Z\colon \text{$|f(v-y)|<6\eps$ for every $f\in F$}\big\}.
\]
By \cite[Lemma~4.1, (b)]{ABHLP},
it suffices to find a neighbourhood $W$ of $z$ in the relative weak topology of $B_Z$
such that $W\subset \lambda U + \nu V$.

For every $f\in F$, there are $f_i\in {X_i}^\ast$, $i=1,2,\dotsc$,
with $\sup_{i\in\N}\|f_i\|=\|f\|=1$ such that
\[
f(w)=\sum_{i=1}^\infty f_i(w_i)
\qquad\text{for every $w=(w_i)_{i=1}^\infty\in Z$.}
\]
Let $x=(x_i)_{i=1}^\infty$, $y=(y_i)_{i=1}^\infty$, $z=(z_i)_{i=1}^\infty$ where $x_i,y_i,z_i\in X_i$.
Observe that $\|z_i\|=\lambda\|x_i\|+\nu\|y_i\|$ for every $i\in\N$.
Define
\[
I_0:=\{i\in\N\colon x_i=y_i=0\}
\qquad\text{and}\qquad
I^+:=\{i\in\N\colon \text{$x_i\not=0$ and $y_i\not=0$}\}
\]
(note that $z_i=0$ for every $i\in I_0$,
and $z_i\not=0$ for every $i\in I^+$ because $\|z_i\|=\lambda\|x_i\|+\nu\|y_i\|$ for every $i\in\N$).
By an argument similar to that in the proof of Lemma \ref{lemma: one may assume that the sets I_A and I_B are empty},
we may (and we will) assume that $I_0\cup I^+=\N$.

Pick an $N\in\N$ such that
\[
\sum_{i=N+1}^\infty\|x_i\|<\eps\min\{\lambda,\nu\}
\qquad\text{and}\qquad
\sum_{i=N+1}^\infty\|y_i\|<\eps\min\{\lambda,\nu\}
\]
(and thus also $\sum_{i=N+1}^\infty\|z_i\|<\eps\min\{\lambda,\nu\}$),
and define
\[
I:=\{1,\dotsc,N\}\setminus I_0
\qquad\text{and}\qquad
I_1:=\{N+1,N+2,\dotsc\}\setminus I_0.
\]
Note that $I^+=I\cup I_1$.

For every $i\in I$, defining
\[
\widehat{x}_i:=\frac{x_i}{\|x_i\|},
\quad
\widehat{y}_i:=\frac{y_i}{\|y_i\|},
\quad
\widehat{z}_i:=\frac{z_i}{\|z_i\|},
\quad
\lambda_i:=\frac{\lambda\|x_i\|}{\|z_i\|},
\quad
\nu_i:=\frac{\nu\|y_i\|}{\|z_i\|},
\]
one has $\lambda_i+\nu_i=1$ and
\[
\widehat{z}_i
=\frac{z_i}{\|z_i\|}
=\frac{\lambda}{\|z_i\|}\,x_i+\frac{\nu}{\|z_i\|}\,y_i
=\frac{\lambda\|x_i\|}{\|z_i\|}\,\widehat{x}_i+\frac{\nu\|y_i\|}{\|z_i\|}\,\widehat{y}_i
=\lambda_i\widehat{x}_i+\nu_i\widehat{y}_i.
\]
For every $i\in I$, defining
\[
U_i:=\bigl\{a\in B_{X_i}\colon \text{$|f_i(a-\widehat{x}_i)|<\eps$ for every $f\in F$}\big\}
\]
and
\[
V_i:=\bigl\{b\in B_{X_i}\colon \text{$|f_i(b-\widehat{y}_i)|<\eps$ for every $f\in F$}\big\},
\]
the sets $U_i$ and $V_i$ are neighbourhoods of, respectively, $\widehat{x}_i$ and $\widehat{y}_i$
in the relative weak topology of $B_{X_i}$,
thus, since the space $X_i$ has property \CWOS{},
there is a neighbourhood $W_i$ of $\widehat{z}_i$ in the relative weak topology of $B_{X_i}$
such that $W_i\subset\lambda_i U_i+\nu_i V_i$.
For every $i\in I$,
choose a finite subset $F_i$ of $B_{X_i^\ast}$
and a real number $d_i>0$ such that $C_i\subset W_i$ where
\[
C_i:=\{c\in B_{X_i}\colon \text{$|\xs(c-\widehat{z}_i)|<d_i$ for every $\xs\in F_i$}\}.
\]

Pick a real number $d>0$ so that $d<\eps\min\{\lambda,\nu\}$.
One may assume that, setting $\Gamma :=\sum_{i\in I}\|z_i\|$,
one has $d<\frac{\Gamma}{2}$,
and $\max\bigl\{\frac{1}{\lambda},\frac{1}{\nu}\bigr\}\cdot\frac{2d}{\Gamma}
     <\min\Bigl\{\frac{\|x_i\|}{\|z_i\|},\frac{\|y_i\|}{\|z_i\|}\Bigr\}$,
$d<\eps\|z_i\|$, and $2d<d_i\|z_i\|$ for every $i\in I$.

There is a neighbourhood $W$ of $z$
in the relative weak topology of $B_Z$ such that,
for every $w=(w_i)_{i=1}^\infty\in W$, one has
\begin{itemize}
\item[(1)]
$\sum\limits_{i=1}^\infty\bigl|\|w_i\|-\|z_i\|\bigr|<d$;
\item[(2)]
$\dfrac{w_i}{\|w_i\|}\in W_i$ for every $i\in I$;
\item[(3)]
$\sum\limits_{i\in I}\|w_i\|>1-2\eps\min\{\lambda,\nu\}$;
\item[(4)]
$\biggl|\dfrac{\|w_i\|}{\|z_i\|}-1\biggr|<\eps$ for every $i\in I$.
\end{itemize}
Indeed, choose an $m\in\N$ with $m\geq N$ so that $\sum_{i=m+1}^\infty\|z_i\|<\frac{d}{6}$.
For every $i\in\{1,\dotsc,m\}$, pick a $\zs_i\in S_{{X_i}^\ast}$ such that $\zs_i(z_i)=\|z_i\|$,
and set $\delta_i:=\frac{d}{6\cdot2^i}$.
Define
\begin{align*}
W:=\bigl\{w=(w_i)_{i=1}^\infty\in B_Z\colon
&\text{$|\zs_i(w_i-z_i)|<\delta_i$ for every $i\in\{1,\dotsc,m\}$, and}\\
&\text{$|\xs(w_i-z_i)|<d$ for every $i\in I$ and every $\xs\in F_i$}\bigr\}.
\end{align*}
Let $w=(w_i)_{i=1}^\infty\in W$ be arbitrary.
Since, for every $i\in\{1,\dotsc,m\}$,
\begin{equation}\label{eq: ||w_i||>=||z_i||-delta_i}
\|w_i\|\geq|\zs_i(w_i)|\geq|\zs_i(z_i)|-|\zs_i(z_i-w_i)|>\|z_i\|-\delta_i,
\end{equation}
one has
\begin{equation*}%\label{eq: sum_(i=m+1)^infty||w_i||<(d)/3}
\begin{aligned}
\sum_{i=m+1}^\infty\|w_i\|
&=\|w\|-\sum_{i=1}^m\|w_i\|
\leq1-\sum_{i=1}^m(\|z_i\|-\delta_i)
=\sum_{i=m+1}^\infty\|z_i\|+\sum_{i=1}^m \delta_i\\
&<\frac{d}{6}+\frac{d}{6}=\frac{d}{3}.
\end{aligned}
\end{equation*}
Define $J^+:=\bigl\{i\in\{1,\dotsc,m\}\colon \|w_i\|\geq\|z_i\|\bigr\}$
and $J^-:=\{1,\dotsc,m\}\setminus J^+$.
Write $\eta_i:=\|w_i\|-\|z_i\|$ for every $i\in\N$.
From \eqref{eq: ||w_i||>=||z_i||-delta_i},
one has $\sum_{i\in J^-}|\eta_i|<\sum_{i\in J^-}\delta_i<\frac{d}{6}$,
thus
\begin{align*}
1&\geq\sum_{i=1}^m\|w_i\|
=\sum_{i=1}^m(\|z_i\|+\eta_i)
=\sum_{i=1}^m\|z_i\|-\sum_{i\in J^-}|\eta_i|+\sum_{i\in J^+}|\eta_i|\\
&>1-\frac{d}{6}-\frac{d}{6}+\sum_{i\in J^+}|\eta_i|
\end{align*}
whence $\sum_{i\in J^+}|\eta_i|<\frac{d}{3}$.
It follows that
\begin{align*}
\sum_{i=1}^\infty|\eta_i|
&=\sum_{i\in J^+}|\eta_i|
  +\sum_{i\in J^-}|\eta_i|
  +\sum_{i=m+1}^\infty\bigl|\|w_i\|-\|z_i\|\bigr|
<\frac{d}{3}+\frac{d}{6}+\sum_{i=m+1}^\infty\bigl(\|w_i\|+\|z_i\|\bigr)\\
&<\frac{d}{3}+\frac{d}{6}+\frac{d}{3}+\frac{d}{6}
=d.
\end{align*}
This establishes (1).
From \eqref{eq: ||w_i||>=||z_i||-delta_i}, it follows that
\begin{align*}
\sum_{i\in I}\|w_i\|
&>\sum_{i\in I}\|z_i\|-\sum_{i\in I}\delta_i
=\sum_{i=1}^N\|z_i\|-\sum_{i\in I}\delta_i
>1-\eps\min\{\lambda,\nu\}-d\\
&>1-2\eps\min\{\lambda,\nu\},
\end{align*}
thus (3) holds.
Now let $i\in I$.
From (1), it follows that $\bigl|\|w_i\|-\|z_i\|\bigr|<d<\eps\|z_i\|$.
This implies (4).
For every $\xs\in F_i$, using (1), one has
\begin{align*}
\biggl|\xs\biggl(\frac{w_i}{\|w_i\|}-\widehat{z}_i\biggr)\biggr|
&=\biggl|\xs\biggl(\frac{w_i}{\|w_i\|}-\frac{z_i}{\|z_i\|}\biggr)\biggr|
\leq\frac{|\xs(w_i-z_i)|}{\|z_i\|}+\biggl|\xs\biggl(\frac{w_i}{\|w_i\|}-\frac{w_i}{\|z_i\|}\biggr)\biggr|\\
&<\frac{d}{\|z_i\|}+\biggl|\frac{1}{\|w_i\|}-\frac{1}{\|z_i\|}\biggr|\|w_i\|
=\frac{d}{\|z_i\|}+\frac{\bigl|\|w_i\|-\|z_i\|\bigr|}{\|z_i\|}\\
&<\frac{d}{\|z_i\|}+\frac{d}{\|z_i\|}
=\frac{2d}{\|z_i\|}
<d_i,
\end{align*}
thus $\frac{w_i}{\|w_i\|}\in W_i$.
This establishes (2).

For every $i\in I$, since $\frac{w_i}{\|w_i\|}\in W_i$ by (2),
there are $a_i\in U_i$ and $b_i\in V_i$ such that $\frac{w_i}{\|w_i\|}=\lambda_i a_i+\nu_i b_i$.

Defining $\eta_i:=\|w_i\|-\|z_i\|$ for every $i\in\N$,
one has $\sum_{i=1}^\infty|\eta_i|<d$ by (1).
Thus, setting,
\[
\kappa:=\sum_{i\in I^+}\frac{\|y_i\|-\|x_i\|}{\|z_i\|}\,\eta_i,
\]
one has
\[
|\kappa|
\leq\max\biggl\{\frac{1}{\lambda},\frac{1}{\nu}\biggr\}\sum_{i\in I^+}|\eta_i|
<\max\biggl\{\frac{1}{\lambda},\frac{1}{\nu}\biggr\}\,d<\eps.
\]
Observe that
\begin{align*}
\rho
:=\sum_{i\in I}\|w_i\|
=\sum_{i\in I}\bigl(\|z_i\|+\eta_i\bigr)
=\Gamma+\sum_{i\in I}\eta_i
\geq\Gamma-\sum_{i\in I}|\eta_i|
>\Gamma-d
>\frac{\Gamma}{2}
\end{align*}
and thus, for every $i\in I$,
\begin{align*}
\frac{|\kappa|}{\rho}
\leq \frac{2|\kappa|}{\Gamma}
< \max\biggl\{\frac{1}{\lambda},\frac{1}{\nu}\biggr\}\cdot\frac{2d}{\Gamma}
<\min\biggl\{\frac{\|x_i\|}{\|z_i\|},\frac{\|y_i\|}{\|z_i\|}\biggr\}.
\end{align*}

Define $u :=(u_i)_{i=1}^\infty$ and $v :=(v_i)_{i=1}^\infty$
where $u_i := v_i := w_i$ if $i\in I_0$, and, for $i\in I^+$,
  \begin{equation*}
    \left\{
      \begin{aligned}
        u_i&:= \frac{\|x_i\|}{\|z_i\|} w_i,\\
        v_i&:= \frac{\|y_i\|}{\|z_i\|} w_i
      \end{aligned}
    \right.
    \quad\text{if $i\in I_1$,}
    \quad\text{and}\quad
    \left\{
      \begin{aligned}
        u_i&:=
        \frac{\|x_i\|}{\|z_i\|}\,\|w_i\|a_i + \nu c_i \\
        v_i&:=
        \frac{\|y_i\|}{\|z_i\|}\,\|w_i\|b_i - \lambda c_i
      \end{aligned}
    \right.
    \quad\text{if $i\in I$.}
  \end{equation*}
where
\[
c_i=\frac{\kappa}{\rho}\,\|w_i\|b_i
\quad\text{if $\kappa\geq0$,}
\qquad\text{and}\qquad
c_i=\frac{\kappa}{\rho}\,\|w_i\|a_i
\quad\text{if $\kappa<0$}.
\]
For every $i\in I$,
observing that  $\|\alpha a_i+\beta b_i\|=\alpha\|a_i\|+\beta\|b_i\|=\alpha+\beta$ whenever $\alpha,\beta\geq0$
(this is because $\|\lambda_i a_i+\nu_i b_i\|=1=\|\lambda_i a_i\|+\|\nu_i b_i\|$), one has
\[
\|u_i\|=\biggl(\frac{\|x_i\|}{\|z_i\|}+\nu\,\frac{\kappa}{\rho}\biggr)\|w_i\|
\qquad\text{and}\qquad
\|v_i\|=\biggl(\frac{\|y_i\|}{\|z_i\|}-\lambda\,\frac{\kappa}{\rho}\biggr)\|w_i\|.
\]
For every $i\in\N$,
one has $\lambda u_i + \nu v_i = w_i$;
thus $\lambda u + \nu v = w$.
It remains to show that $u\in U$ and $v\in V$.

One has
\begin{align*}
\|u\|
&=\sum_{i\in I_0}\|w_i\|+\sum_{i \in I_1} \frac{\|x_i\|}{\|z_i\|} \|w_i\|
         +\sum_{i \in I} \left(\frac{\|x_i\|}{\|z_i\|}
         + \nu\,\frac{
          \kappa}{\rho}\right) \|w_i\| \\
&=\sum_{i\in I_0}\|w_i\|+\sum_{i\in I^+}\frac{\|x_i\|}{\|z_i\|} \|w_i\| + \nu\kappa \\
&=\sum_{i\in I_0}\eta_i+\sum_{i\in I^+}\frac{\|x_i\|}{\|z_i\|}(\|z_i\| + \eta_i)
         +\nu\sum_{i\in I^+}\frac{\|y_i\|-\|x_i\|}{\|z_i\|}\,\eta_i\\
&=\sum_{i\in I_0}\eta_i+\sum_{i\in I^+}\|x_i\|
         +\sum_{i\in I^+}\frac{\lambda\|x_i\|+\nu\|y_i\|}{\|z_i\|}\,\eta_i
=1 + \sum_{i=1}^\infty\eta_i\\
&\leq1
\end{align*}
because
\[
\sum_{i=1}^\infty\eta_i
=\sum_{i=1}^\infty\|w_i\|-\sum_{i=1}^\infty\|z_i\|=\|w\|-\|z\|=\|w\|-1\leq0.
\]
By symmetry, also $\|v\|\leq1$.

Now let $f\in F$ be arbitrary. One has
\begin{align*}
|f(u-x)|
&=\biggl|\sum_{i=1}^\infty f_i(u_i-x_i)\biggr|
\leq\sum_{i\in I_0\cup I_1}(\|u_i\|+\|x_i\|)+\sum_{i\in I}|f_i(u_i-x_i)|.
\end{align*}
Since $w\in W$ with $w=\lambda u+\nu v$, where $\|u\|\leq1$ and $\|v\|\leq1$, by (3),
\[
1-2\eps\lambda
<\sum_{i\in I}\|w_i\|
\leq\lambda\sum_{i\in I}\|u_i\|+\nu\sum_{i\in I}\|v_i\|
\leq\lambda\sum_{i\in I}\|u_i\|+\nu,
\]
whence
\[
\lambda\sum_{i\in I}\|u_i\|
>1-2\eps\lambda-\nu
=\lambda-2\eps\lambda;
\]
thus $\sum_{i\in I}\|u_i\|>1-2\eps$
and hence
\[
\sum_{i\in I_0\cup I_1}\|u_i\|
\leq1-\sum_{i\in I}\|u_i\|
<1-(1-2\eps)
=2\eps.
\]
Thus, keeping in mind that $\sum_{i\in I_0\cup I_1}\|x_i\|=\sum_{i=N+1}^\infty\|x_i\|<\eps$, one has
\[
\sum_{i\in I_0\cup I_1}(\|u_i\|+\|x_i\|)
=\sum_{i\in I_0\cup I_1}\|u_i\|+\sum_{i\in I_0\cup I_1}\|x_i\|
<2\eps+\eps=3\eps.
\]
For every $i\in I$,
\begin{align*}
u_i-x_i
&=\|x_i\|\biggl(\frac{\|w_i\|}{\|z_i\|}\,a_i-\widehat{x}_i\biggr)+\nu c_i
=\|x_i\|\left(\biggl(\frac{\|w_i\|}{\|z_i\|}-1\biggr)a_i+(a_i-\widehat{x}_i)\right)+\nu c_i,
\end{align*}
thus
\begin{align*}
\sum_{i\in I}|f_i(u_i-x_i)|
&\leq\sum_{i\in I}\left(\|x_i\|\biggl(\biggl|\frac{\|w_i\|}{\|z_i\|}-1\biggr||f_i(a_i)|
                                       +|f_i(a_i-\widehat{x}_i)|\biggr)
                        +\nu \|c_i\|\right)\\
&=\sum_{i\in I}\biggl(\|x_i\|(\eps+\eps)+\nu\,\frac{|\kappa|}{\rho}\,\|w_i\|\biggr)
=2\eps\sum_{i\in I}\|x_i\|+\nu\,\frac{|\kappa|}{\rho}\sum_{i\in I}\|w_i\|\\
&<3\eps.
\end{align*}
It follows that $|f(u-x)|<6\eps$; thus $u\in U$.
By symmetry, also $v\in V$,
and the proof is complete.
\end{proof}

%%%%%%%%%%%%%%%%%%%%%%%%%%%%%%%%%%%%%%%%%%%%%%%%%%%%%%%%%%%%%%%%%%%%%%
\section{Proof of  Theorem~\ref{thm: L_1(mu,X) has property CWO-S if X is wUR}}
%%%%%%%%%%%%%%%%%%%%%%%%%%%%%%%%%%%%%%%%%%%%%%%%%%%%%%%%%%%%%%%%%%%%%%

\begin{proof}[Proof of  Theorem~\ref{thm: L_1(mu,X) has property CWO-S if X is wUR}]
By Proposition \ref{prop: 1-sum of Banach spaces with property CWO-S has property CWO-S},
one may assume that the measure $\mu$ is finite.

Let $U$ and $V$ be relatively weakly open subsets of $B_{L_1(\mu,X)}$,
and let $x \in U$, $y \in V$, and $\lambda, \nu>0$ with $\lambda + \nu = 1$
be such that $\|\lambda x + \nu y \|=1$.
By \cite[Lemma~4.1, (b)]{ABHLP}, it suffices to find
a neighbourhood~$W$ of $z := \lambda x + \nu y$ in the relative weak topology of $B_{L_1(\mu,X)}$
such that $W\subset \lambda U + \nu V$.
By Lemma~\ref{lemma: one may assume that the sets I_A and I_B are empty},
one may assume that, defining $C:=\bigl\{\omega\in\Omega\colon \text{$x(\omega)\not=0$ and $y(\omega)\not=0$}\bigr\}$,
one has $\Omega=C\cup\bigl\{\omega\in\Omega\colon \text{$x(\omega)=y(\omega)=0$}\bigr\}$.

Pick an $m\in\N$ so that $\frac{2}{\lambda}\leq m$ and $\frac{2}{\nu}\leq m$.
Identify functionals in $L_1(\mu,X)^\ast$ with functions in $L_\infty(\mu,\Xs)$ in the canonical way (this can be done by  \cite[page 98, Theorem 1]{DU} because the dual space $\Xs$ has the Radon--Nikod\'ym property by a result of H\'ajek \cite[Theorem 1]{H}).
Keeping in mind that each function in $L_\infty(\mu,\Xs)$ is essentially separably valued, there are pairwise disjoint sets $E_j\in\Sigma$ and finite subsets $G_j$ of $B_{\Xs}$, $j=1,2,\dotsc$,
and an $\eps>0$ with $\eps<1$ such that $\bigcup_{j=1}^\infty E_j=\Omega$ and, defining a seminorm $p$ on $L_1(\mu,X)$ by
\[
p(w):=\sum_{j=1}^\infty\max_{\xs\in G_j}\biggl|\int_{E_j}\xs(w)\,d\mu\biggr|,
\]
one has $U'\subset U$ and $V'\subset V$ where
$U':=\bigl\{u \in B_{L_1(\mu,X)}\colon p(u-x)<(m+4)\eps\bigr\}$
and $V':=\bigl\{v\in B_{L_1(\mu,X)}\colon p(v-y)<(m+4)\eps\bigr\}$.

Let a real number $\delta>0$ be such that $\int_E\|x\|\,d\mu<\eps$ and $\int_E\|y\|\,d\mu<\eps$
(and thus also $\int_E\|z\|\,d\mu<\eps$) whenever $E\in\Sigma$ satisfies $\mu(E)<3\delta$.

%Define
%\begin{gather*}
%A:=\bigl\{\omega\in\Omega\colon \text{$x(\omega)\not=0$ and $y(\omega)=0$}\bigr\},\quad
%B:=\bigl\{\omega\in\Omega\colon \text{$x(\omega)=0$ and $y(\omega)\not=0$}\bigr\},\\
%C:=\bigl\{\omega\in\Omega\colon \text{$x(\omega)\not=0$ and $y(\omega)\not=0$}\bigr\}.
%%D&:=\bigl\{\omega\in\Omega\colon \text{$x(\omega)=y(\omega)=0$}\bigr\},
%\end{gather*}

Pick an $N\in\N$ and real numbers $\sigma,\tau>0$
so that $\mu\Bigl(\bigcup_{j=N+1}^\infty E_j\Bigr)<\delta$
and
%$\mu(A\setminus A')<\delta$, $\mu(B\setminus B')<\delta$,
$\mu(C\setminus C')<\delta$ where
\begin{gather*}
%A':=\bigl\{\omega\in A\colon \sigma<\|x(\omega)\|\leq\tau \bigr\}, \quad
%B':=\bigl\{\omega\in B\colon \sigma<\|y(\omega)\|\leq\tau \bigr\},\\
C':=\bigl\{\omega\in C\colon \text{$\sigma<\|\phi(\omega)\|\leq\tau$ for every $\phi\in\{x,y,z\}$}\bigr\}.
\end{gather*}
Write $G:=\bigcup_{j=1}^N G_j$
and let $\theta\in(0,1)$ be a real number
satisfying the conclusions of Lemma~\ref{lemma: wUR}
with $\xi:=\min\bigl\{\lambda\,\frac{\sigma}{\tau},\nu\,\frac{\sigma}{\tau}\bigr\}$ for every $\xs\in G$.

For every $j\in\{1,\dotsc,N\}$, by applying Lemma \ref{lemma: L_1(mu,X)}, (a),
with $E=E_j\cap C'$ and $\eps=\min\bigl\{1-\theta,\frac{\delta}{N}\bigr\}$,
one obtains $n_j\in\{0\}\cup\N$
and pairwise disjoint measurable subsets $E^j_0,E^j_1,\dotsc,E^j_{n_j}$ of $E_j\cap C'$
with $\bigcup_{k=0}^{n_j}E^j_k=E_j\cap C'$ such that $\mu(E^j_0)<\frac{\delta}{N}$,
and $\mu(E^j_k)>0$ and $\bigl\|\int_{E^j_k}z\,d\mu\bigr\|>\theta\int_{E^j_k}\|z\|\,d\mu$
for every $k\in\{1,\dotsc,n_j\}$.
Relabel the sets $E^j_1,\dotsc,E^j_{n_j}$, $j=1,\dotsc,N$,
to become $D_2,\dotsc,D_n$ where $n=n_1+\dotsb+n_N+1$.
Set $D:=\bigcup_{i=2}^n D_i$, and $D_0:=\bigl\{\omega\in\Omega\setminus D\colon z(\omega)=0\bigr\}$
and $D_1:=\bigl\{\omega\in\Omega\setminus D\colon z(\omega)\not=0\bigr\}$.
Observe that $\int_{D_0}\|x\|\,d\mu=0$ and $\int_{D_0}\|y\|\,d\mu=0$
(because otherwise one would have $\|z\|<\lambda\|x\|+\nu\|y\|=1$).
Since $\mu(D_1)<3\delta$,
one has $\int_{D_1}\|x\|\,d\mu<\eps$ and $\int_{D_1}\|y\|\,d\mu<\eps$,
and thus also $\int_{D_1}\|z\|\,d\mu<\eps$.
Since $\int_{D_0\cup D_1}\|z\|\,d\mu=\int_{D_1}\|z\|\,d\mu<\eps<1$,
it follows that, indeed, the collection $\{D_2,\dotsc,D_n\}$ is non-empty.

Observe that, for every $w\in L_1(\mu,X)$, one has
$p(w)\leq q(w)+\int_{D_0\cup D_1}\|w\|\,d\mu$ where
\[
q(w):=\sum_{i=2}^n\max_{\xs\in G}\biggl|\int_{D_i}\xs(w)\,d\mu\biggr|.
\]
Since $\int_{D_0\cup D_1}\|x\|\,d\mu<\eps$ and $\int_{D_0\cup D_1}\|y\|\,d\mu<\eps$, one has
\begin{equation}\label{eq: p(u-x)=<q(u-x)+int_(D_0)||u||dmu+eps}
p(u-x) \leq q(u-x) + \int_{D_0\cup D_1}\|u\|\,d\mu + \eps
\quad\text{for every $u\in L_1(\mu,X)$}
\end{equation}
and $p(v-y) \leq q(y-y) + \int_{D_0\cup D_1}\|v\|\,d\mu + \eps$ for every $v\in L_1(\mu,X)$.

For every $i\in \{0,1,\dotsc,n\}$, set
\begin{equation*}
a_i := \int_{D_i}x\,d\mu,
\quad
b_i := \int_{D_i}y\,d\mu,
\quad\text{and}\quad
c_i := \int_{D_i}z\,d\mu
\end{equation*}
and
\begin{equation*}
\alpha_i := \int_{D_i}\|x\|\,d\mu,
\quad
\beta_i := \int_{D_i}\|y\|\,d\mu,
\quad\text{and}\quad
\gamma_i := \int_{D_i}\|z\|\,d\mu.
\end{equation*}
Notice that $\lambda a_i + \nu b_i = c_i$ and $\lambda\alpha_i + \nu \beta_i = \gamma_i$
(the latter is because if $\lambda\alpha_i + \nu \beta_i > \gamma_i$ for some $i\in \{0,1,\dotsc,n\}$,
then one would have $\|z\|<\lambda\|x\|+\nu\|y\|=1$).
Notice also that $\frac{\alpha_i}{\gamma_i}\geq\frac{\sigma}{\tau}$
and $\frac{\beta_i}{\gamma_i}\geq\frac{\sigma}{\tau}$ for every $i\in \{2,\dotsc,n\}$
(this is because $\alpha_i\geq\sigma\mu(D_i)$ and $\beta_i\geq\sigma\mu(D_i)$,
and $\gamma_i\leq\tau\mu(D_i)$ for every $i\in \{2,\dotsc,n\}$).
One also has $\bigl\|\frac{c_i}{\gamma_i}\bigr\|>\theta$ for every $i\in \{2,\dotsc,n\}$.

Whenever $i\in \{2,\dotsc,n\}$, observing that
\[
\lambda\frac{\alpha_i}{\gamma_i}\,\frac{a_i}{\alpha_i}
+\nu\frac{\beta_i}{\gamma_i}\,\frac{b_i}{\beta_i}
=\frac{c_i}{\gamma_i}
\]
where $\lambda\frac{\alpha_i}{\gamma_i}+\nu\frac{\beta_i}{\gamma_i}=1$
with $\min\bigl\{\lambda\frac{\alpha_i}{\gamma_i},\nu\frac{\beta_i}{\gamma_i}\bigr\}\geq\xi$,
one has $\bigl|\xs\bigl(\frac{a_i}{\alpha_i}-\frac{c_i}{\gamma_i}\bigr)\bigr|<\eps$
and $\bigl|\xs\bigl(\frac{b_i}{\beta_i}-\frac{c_i}{\gamma_i}\bigr)\bigr|<\eps$ for every $\xs\in G$ by the choice of $\theta$,
and thus $\bigl|\xs\bigl(a_i-\frac{\alpha_i}{\gamma_i}\,c_i\bigr)\bigr|<\eps\alpha_i$
and $\bigl|\xs\bigl(b_i-\frac{\beta_i}{\gamma_i}\,c_i\bigr)\bigr|<\eps\beta_i$ for every $\xs\in G$.

Define $I^+:=\bigl\{i\in\{1,\dotsc,n\}\colon \gamma_i>0\bigr\}$
and $I^0:=\{0,1,\dotsc,n\}\setminus I^+$
(that is, $I^0:=\{0,1\}$ and $I^+:=\{2,\dotsc,n\}$ if $\mu(D_1)=0$,
and $I^0:=\{0\}$ and $I^+:=\{1,\dotsc,n\}$ if $\mu(D_1)>0$).

Pick a real number $d>0$ so that $\max\bigl\{\frac{1}{\lambda},\frac{1}{\nu}\bigr\}\,d<\eps$.
One may assume that, setting $\Gamma :=\sum_{i=2}^n\gamma_i$,
one has $d<\frac{\Gamma}{2}$
and $\max\bigl\{\frac{1}{\lambda},\frac{1}{\nu}\bigr\}\cdot\frac{2d}{\Gamma}<\frac{\sigma}{\tau}$.
By Lemma \ref{lemma: L_1(mu,X)}, (b), there is a neighbourhood $W_0$ of $z$
in the relative weak topology of $B_{L_1(\mu, X)}$ such that, whenever $w\in W_0$ and $i\in\{0,1,\dotsc,n\}$,
\begin{equation}\label{eq: gamma_i - d/(n+1) < int_(D_i)||w|dmu| < gamma_i + d/(n+1)}
\gamma_i-\frac{d}{n+1}
<\int_{D_i}\|w\|\,d\mu
<\gamma_i+\frac{d}{n+1}.
\end{equation}
Define
\begin{equation*}
W:=
W_0\cap\bigcap_{i=2}^n\biggl\{w\in B_{L_1(\mu,X)}\colon
\max_{\xs\in G}\biggl|\int_{D_i}\xs(w-z)\,d\mu \biggr|<\eps\gamma_i\biggr\}.
\end{equation*}
Let $w\in W$ be arbitrary.

For every $i\in\{0,1,\dotsc,n\}$, set $w_i=\chi_{D_i}w$ and $\eta_i:=\|w_i\|-\gamma_i$.
By \eqref{eq: gamma_i - d/(n+1) < int_(D_i)||w|dmu| < gamma_i + d/(n+1)},
one has $\sum_{i=0}^n|\eta_i|<d$.
Thus, setting,
\[
\kappa:=\sum_{i\in I^+}\frac{\beta_i-\alpha_i}{\gamma_i}\,\eta_i,
\]
one has
\[
|\kappa|
\leq\max\biggl\{\frac{1}{\lambda},\frac{1}{\nu}\biggr\}\sum_{i\in I^+}|\eta_i|
<\max\biggl\{\frac{1}{\lambda},\frac{1}{\nu}\biggr\}\,d<\eps.
\]
Observe that
\begin{align*}
\rho
:=\sum_{i=2}^n\|w_i\|
=\sum_{i=2}^n\bigl(\gamma_i+\eta_i\bigr)
=\Gamma+\sum_{i=2}^n\eta_i
\geq\Gamma-\sum_{i=2}^n|\eta_i|
>\Gamma-d
>\frac{\Gamma}{2}
\end{align*}
and thus, for every $i\in\{2,\dotsc,n\}$,
\begin{align*}
\frac{|\kappa|}{\rho}
\leq \frac{2|\kappa|}{\Gamma}
< \max\biggl\{\frac{1}{\lambda},\frac{1}{\nu}\biggr\}\cdot\frac{2d}{\Gamma}
<\frac{\sigma}{\tau}
\leq\min\biggl\{\frac{\alpha_i}{\gamma_i},\frac{\beta_i}{\gamma_i}\biggr\}.
\end{align*}

Define $u := \sum_{i=0}^n u_i$ and $v := \sum_{i=0}^n v_i$
where $u_i := v_i := w_i$ if $i\in I^0$, and, for $i\in I^+$,
  \begin{equation*}
    \left\{
      \begin{aligned}
        u_i&:= \frac{\alpha_i}{\gamma_i} w_i,\\
        v_i&:= \frac{\beta_i}{\gamma_i} w_i
      \end{aligned}
    \right.
    \quad\text{if $i=1$,}
    \quad\text{and}\quad
    \left\{
      \begin{aligned}
        u_i&:=
        \left(\frac{\alpha_i}{\gamma_i} + \nu\,\frac{\kappa}{\rho}\right) w_i, \\
        v_i&:=
        \left(\frac{\beta_i}{\gamma_i} - \lambda\,\frac{\kappa}{\rho}\right) w_i
      \end{aligned}
    \right.
    \quad\text{if $i\in\{2,\dotsc,n\}$.}
  \end{equation*}
For every $i\in \{0,1,\dotsc,n\}$,
one has $\lambda u_i + \nu v_i = w_i$;
thus $\lambda u + \nu v = \sum_{i=0}^n w_i = w$.
It remains to show that $u\in U$ and $v\in V$.

One has
\begin{align*}
\|u\|
&=\|u_0\|+\|u_1\| +\sum_{i=2}^n \left(\frac{\alpha_i}{\gamma_i} + \nu\,\frac{\kappa}{\rho}\right)\|w_i\| \\
&=\sum_{i\in I^0}\|w_i\|+\sum_{i\in I^+}\frac{\alpha_i}{\gamma_i} \|w_i\| + \nu\kappa \\
&=\sum_{i\in I^0}\eta_i+\sum_{i\in I^+}\frac{\alpha_i}{\gamma_i}(\gamma_i + \eta_i)
         +\nu\sum_{i\in I^+}\frac{\beta_i-\alpha_i}{\gamma_i}\,\eta_i\\
&=\sum_{i\in I^0}\eta_i+\sum_{i\in I^+}\alpha_i
         +\sum_{i\in I^+}\frac{\lambda\alpha_i+\nu\beta_i}{\gamma_i}\,\eta_i
=1 + \sum_{i=0}^n\eta_i\\
&\leq1
\end{align*}
because
\[
\sum_{i=0}^n\eta_i
=\sum_{i=0}^n\|w_i\|-\sum_{i=0}^n\gamma_i=\|w\|-\|z\|=\|w\|-1\leq0.
\]
By symmetry, also $\|v\|\leq1$.

Now let $\xs\in G$ be arbitrary.
For every $i\in\{2,\dotsc,n\}$, one has
\begin{align*}
\biggl|\int_{D_i}\xs\biggl(\frac{\alpha_i}{\gamma_i}\,w_i-x\biggr)\,d\mu\biggr|
&=\biggl|\int_{D_i}\xs\biggl(\frac{\alpha_i}{\gamma_i}\,w_i
                               -\frac{\alpha_i}{\gamma_i}\,z+\frac{\alpha_i}{\gamma_i}\,z-x\biggr)\,d\mu\biggr|\\
&\leq\frac{\alpha_i}{\gamma_i}\biggl|\int_{D_i}\xs(w-z)\,d\mu\biggr|
      +\biggl|\xs\biggl(\int_{D_i}\biggl(\frac{\alpha_i}{\gamma_i}z-x\biggr)\,d\mu\biggr)\biggr|\\
%&\leq\frac{\alpha_i}{\gamma_i}\,\eps\gamma_i
%      \roosavalem{+\biggl\|\int_{D_i}\biggl(\frac{\alpha_i}{\gamma_i}z-x\biggr)\,d\mu\biggr\|}
%      \hallvalem{+\biggl|\xs\biggl(\frac{\alpha_i}{\gamma_i}\,c_i-a_i\biggr)\biggr|}\\
&<\eps\alpha_i +\biggl|\xs\biggl(\frac{\alpha_i}{\gamma_i}\,c_i-a_i\biggr)\biggr|\\
&<2\eps\alpha_i
\end{align*}
and thus
\begin{align*}
\biggl|\int_{D_i}\xs(u-x)\,d\mu\biggr|
&=\biggl|\int_{D_i}\xs\biggl(\frac{\alpha_i}{\gamma_i}\,w_i+\nu\,\frac{\kappa}{\rho}\,w_i-x\biggr)\,d\mu\biggr|\\
&\leq\biggl|\int_{D_i}\xs\biggl(\frac{\alpha_i}{\gamma_i}\,w_i-x\biggr)\,d\mu\biggr|
    +\frac{|\kappa|}{\rho}\int_{D_i}\|w_i\|\,d\mu\\
%&<\eps\alpha_i\roosavalem{+\biggl\|\frac{\alpha_i}{\gamma_i}\,c_i-a_i\biggr\|}
%              \hallvalem{+\biggl|\xs\biggl(\frac{\alpha_i}{\gamma_i}\,c_i-a_i\biggr)\biggr|}
%    +\eps\,\frac{\|w_i\|}{\rho}\\
&<2\eps\alpha_i+\eps\,\frac{\|w_i\|}{\rho}.
\end{align*}
It follows that
\begin{align*}
q(u-x)
%&=\sum_{i=2}^n\max_{\xs\in G}\biggl|\int_{D_i}\xs(u-x)\biggr|\\
%&=\sum_{i\in I_A}\max_{\xs\in G}\biggl|\int_{D_i}\xs(u-x)\biggr|
%+\sum_{i\in I_C}\max_{\xs\in G}\biggl|\int_{D_i}\xs(u-x)\biggr|\\
&<\sum_{i=2}^n\biggl(2\eps\alpha_i+\eps\,\frac{\|w_i\|}{\rho}\biggr)
\leq2\eps\sum_{i=2}^n\alpha_i+\frac{\eps}{\rho}\sum_{i=2}^n\|w_i\|
\leq3\eps.
\end{align*}
Since $w\in W_0$ with $w=\lambda u+\nu v$, where $\|u\|\leq1$ and $\|v\|\leq1$,
one has (recall that $D=\bigcup_{i=2}^n D_i$)
\begin{align*}
1-2\eps
&<1-\int_{D_0\cup D_1}\|z\|\,d\mu-\eps
=\int_{D}\|z\|\,d\mu-\eps\\
&<\int_{D}\|w\|\,d\mu
\leq\lambda\int_{D}\|u\|\,d\mu+\nu\int_{D}\|v\|\,d\mu
\leq\lambda\int_{D}\|u\|\,d\mu+\nu
\end{align*}
whence
\begin{align*}
\lambda\int_{D}\|u\|\,d\mu
>1-2\eps-\nu=\lambda-2\eps
\end{align*}
and thus $\int_{D}\|u\|\,d\mu>1-\frac{2\eps}{\lambda}\geq1-m\eps$; hence
\begin{align*}
\int_{D_0\cup D_1}\|u\|\,d\mu
&=\|u\|-\int_{D}\|u\|\,d\mu
<1-(1-m\eps)
=m\eps.
\end{align*}
By \eqref{eq: p(u-x)=<q(u-x)+int_(D_0)||u||dmu+eps},
it follows that $p(u-x)<3\eps+m\eps+\eps=(m+4)\eps$;
thus $u\in U'\subset U$.
By symmetry, also $v \in V$,
and the proof is complete.
\end{proof}

% ----------------------------------------------------------------
%\bibliographystyle{amsplain}
\bibliographystyle{amsplain_abbrv}
\bibliography{references}

\providecommand{\bysame}{\leavevmode\hbox to3em{\hrulefill}\thinspace}
\providecommand{\MR}{\relax\ifhmode\unskip\space\fi MR }
\providecommand{\MRhref}[2]{%
  \href{http://www.ams.org/mathscinet-getitem?mr=#1}{#2}
}
\providecommand{\href}[2]{#2}
\begin{thebibliography}{10}

\bibitem{AL}
T.~A. Abrahamsen and V.~Lima, \emph{Relatively weakly open convex combinations of slices}, Proc. Amer. Math. Soc. \textbf{146} (2018), no.~10, 4421--4427. \MR{3834668}

\bibitem{ABHLP}
T.~A. Abrahamsen, J.~B. Guerrero, R.~Haller, V.~Lima, and M.~P\~{o}ldvere, \emph{Banach spaces where convex combinations of relatively weakly open subsets of the unit ball are relatively weakly open}, Studia Math. \textbf{250} (2020), no.~3, 297--320. \MR{4034749}

\bibitem{CP}
A.~Clausing and S.~Papadopoulou, \emph{Stable convex sets and extremal operators}, Math. Ann. \textbf{231} (1977/78), no.~3, 193--203. \MR{467249}

\bibitem{DU}
J.~Diestel and J.~J. Uhl, Jr., \emph{Vector measures}, Mathematical Surveys, No. 15, American Mathematical Society, Providence, R.I., 1977, With a foreword by B. J. Pettis. \MR{0453964}

\bibitem{GGMS}
N.~Ghoussoub, G.~Godefroy, B.~Maurey, and W.~Schachermayer, \emph{Some topological and geometrical structures in {B}anach spaces}, Mem. Amer. Math. Soc. \textbf{70} (1987), no.~378, iv+116. \MR{912637}

\bibitem{GP}
J.~B. Guerrero and F.~J. Fern\'{a}ndez-Polo, \emph{Relatively weakly open convex combinations of slices and scattered {$\rm C^*$}-algebras}, Mediterr. J. Math. \textbf{17} (2020), no.~4, Paper No. 117, 23. \MR{4116184}

\bibitem{H}
P.~H\'{a}jek, \emph{Dual renormings of {B}anach spaces}, Comment. Math. Univ. Carolin. \textbf{37} (1996), no.~2, 241--253. \MR{1398999}

\bibitem{HKP}
R.~Haller, P.~Kuuseok, and M.~P\~{o}ldvere, \emph{On convex combinations of slices of the unit ball in {B}anach spaces}, Houston J. Math. \textbf{45} (2019), no.~4, 1153--1168. \MR{4102873}

\bibitem{Kadets_2019}
V.~Kadets, \emph{A remark to the property {CWO} from the paper {\tt ar{\uppercase{x}}iv:1806.10693}}, private communication, January 14, 2019.

\bibitem{L-PM}
G.~L\'{o}pez-P\'{e}rez and R.~Medina, \emph{A characterization of the weak topology in the unit ball of purely atomic {$L_1$} preduals}, J. Math. Anal. Appl. \textbf{514} (2022), no.~2, Paper No. 126311, 14. \MR{4422399}

\end{thebibliography}

\end{document}